\documentclass[10pt]{article}

\usepackage[english]{babel}
\usepackage{float}
\usepackage{latexsym}
\usepackage{amssymb}
\usepackage{amsmath}
\usepackage{amsthm}
\usepackage{amsfonts}

\usepackage{pstricks}
\usepackage{lineno}

\usepackage{graphicx}
\graphicspath{{figures/}}

\newtheorem{theorem}{Theorem}[section]
\newtheorem{cor}[theorem]{Corollary}

\newtheorem{lemma}[theorem]{Lemma}

\theoremstyle{definition}
\newtheorem{defini}[theorem]{Definition}
\theoremstyle{remark}

\numberwithin{equation}{section}

\title{Geometric achromatic and pseudoachromatic indices}

\author{
O. Aichholzer \footnotemark[2]
\and G. Araujo-Pardo \footnotemark[3]
\and N. Garc{\'i}a-Col{\'i}n \footnotemark[3]
\and T. Hackl \footnotemark[2]
\and D. Lara \footnotemark[4]
\and C. Rubio-Montiel \footnotemark[3]
\and J. Urrutia \footnotemark[3]
 }

\begin{document}
\maketitle

\def\thefootnote{\fnsymbol{footnote}}
\footnotetext[2]{Institute for Software Technology, Graz University of
  Technology, Austria, {\tt [oaich|thackl]@ist.tugraz.at}.}
\footnotetext[3]{Instituto de Matem{\'a}ticas, Universidad Nacional Aut{\'o}noma de M{\'e}xico, Mexico,
{\tt [garaujo|garcia|christian|urrutia]@matem.unam.mx}}
\footnotetext[4]{Departament de Matem{\`a}tica Aplicada II, Universitat Polit{\`e}cnica de Catalunya, Spain,
{\tt maria.dolores.lara@upc.edu}.}

\renewcommand{\thefootnote}{\arabic{footnote}}

\begin{abstract}
  The pseudoachromatic index of a graph is the maximum number of
  colors that can be assigned to its edges, such that each pair of
  different colors is incident to a common vertex. If for each vertex
  its incident edges have different color, then this maximum is known
  as achromatic index. Both indices have been widely studied.
  A geometric graph is a graph drawn in the plane such that its vertices are
  points in general position, and its edges are straight-line
  segments. In this paper we extend the notion of pseudoachromatic and achromatic indices for
  geometric graphs, and present results for complete geometric graphs.
  In particular, we show that for $n$ points in convex position the achromatic index and the pseudoachromatic index of the complete geometric graph are $\lfloor \tfrac{n^2+n}{4} \rfloor$.
\end{abstract}

\section{Introduction}

A \emph{vertex coloring} with $k$ colors of a simple graph $G$, is a
surjective function that assigns to each vertex of $G$ a color from
the set $\{1, 2, \ldots, k\}$. A coloring is \emph{proper} if any two
adjacent vertices have different color, and it is \emph{complete} if
every pair of colors appears on at least one pair of adjacent
vertices. The \emph{chromatic number $\chi(G)$ of $G$} is the smallest
number $k$ for which there exists a proper coloring of $G$ using $k$
colors. It is not hard to see that any proper coloring of $G$ with
$\chi(G)$ colors is a complete coloring. The \emph{achromatic number}
$\alpha(G)$ of $G$ is the biggest number $k$ for which there exists a
proper and complete coloring of $G$ using $k$ colors. The
\emph{pseudoachromatic number} $\psi(G)$ of $G$ is the biggest number
$k$ for which there exists a complete coloring of $G$ using $k$
colors. Clearly we have that \[\chi(G) \leq \alpha(G) \leq \psi(G).\]

The achromatic number was introduced by Harary, Hedetniemi and Prins
in 1967~\cite{MR0272662}; the pseudoachromatic number was introduced
by Gupta in 1969 \cite{MR0256930}. Several authors have studied these
parameters, and it turns out that the exact determination of the
numbers is quite difficult. For details see \cite{MR2778722,
  MR2450569} and the references therein.

The \emph{chromatic index} $\chi_1(G)$, \emph{achromatic index}
$\alpha_1(G)$ and \emph{pseudoachromatic index} $\psi_1(G)$ of $G$,
are defined respectively as the chromatic number, achromatic number
and pseudoachromatic number of the line graph $L(G)$ of
$G$. Notationally, \mbox{$\chi_1(G) = \chi(L(G))$}, \mbox{$\alpha_1(G)
  = \alpha(L(G))$} and \mbox{$\psi_1(G) = \psi(L(G))$}.

A central topic in Graph Theory is to study the behavior of any
parameter in complete graphs. For instance, the authors of \cite{
  MR2778722} and \cite{MR543176} determined the exact value of
$\alpha_1$ and $\psi_1$ for some specific complete graphs. In this
paper we extend the notion of pseudoachromatic and achromatic indices
to geometric graphs, and present upper and lower bounds for the case
of complete geometric graphs.

The next section introduces geometric graphs, generalizes the different numbers and
indices and provides some basic relations for them. In Section~\ref{sec:convex} we consider the complete
graph where the vertex set is a set of points in convex position, and show lower and upper
bounds of the indices. Finally, we generalize this considerations to point sets in general
position in Section~\ref{sec:general} and give bounds for the geometric pseudoachromatic index.

\section{Preliminaries}

Throughout this paper we assume that all sets of points in the plane are in general
position, that is, no three points are on a common line. Let
$G=(V,E)$ be a simple graph. A \emph{geometric embedding} of G is an injective
function that maps $V$ to a set $S$ of points in the plane, and $E$ to a set of
(possibly crossing) straight-line segments whose endpoints belong to $S$. A
\emph{geometric graph} $\mathsf{G}$ is the image of a particular geometric
embedding of G. For brevity we refer to the points in $S$ as
vertices of $\mathsf{G}$, and to the straight-line segments connecting two points in
$S$ as edges of $\mathsf{G}$. Please note that any set of
points in the plane induces a complete geometric graph. We say that
two edges of $\mathsf{G}$ \emph{intersect} if they have a common
endpoint or they cross. Two edges are \emph{disjoint} if they do not
intersect. A coloring of the edges of $\mathsf{G}$ is \emph{proper} if
every pair of edges of the same color are disjoint. A coloring is
\emph{complete} if each pair of colors appears on at least one pair of
intersecting edges.

The \textit{chromatic index} $\chi_1(\mathsf{G})$ of $\mathsf{G}$ is
the smallest number $k$ for which there exists a proper coloring of the
edges of $\mathsf{G}$ using $k$ colors. The \textit{achromatic index}
$\alpha_1(\mathsf{G})$ of $\mathsf{G}$ is the biggest number $k$ for
which there exists a complete and proper coloring of the edges of
$\mathsf{G}$ using $k$ colors. The \textit{pseudoachromatic index}
$\psi_1(\mathsf{G})$ of $\mathsf{G}$ is the biggest number $k$ for
which there exists a complete coloring of the edges of $\mathsf{G}$
using $k$ colors.

We extend these definitions to graphs in the following way. Let $G$ be
a graph. The \textit{geometric chromatic index} $\chi_g(G)$ of $G$ is
the largest value $k$ for which a geometric graph $\mathsf{H}$ of
$G$ exists, such that $\chi_1(\mathsf{H})=k$. Likewise, the \textit{geometric
  achromatic index} $\alpha_g(G)$ and the \textit{geometric
  pseudoachromatic index} $\psi_g(G)$ of $G$, are defined as the
smallest value $k$ for which a geometric graph $\mathsf{H}$ of $G$
exists such that $\alpha_1(\mathsf{H})=k$ and $\psi_1(\mathsf{H})=k$,
respectively.

From the above definitions we get for graphs

\begin{equation}\chi_1(G)\leq\chi_g(G)\end{equation}

\begin{equation}\chi_1(G)\leq\alpha_1(G)\leq\psi_1(G)\leq\psi_g(G)\end{equation}

\begin{equation}\chi_1(G)\leq\alpha_1(G)\leq\alpha_g(G)\leq\psi_g(G)\end{equation}

and for geometric graphs we obtain

\begin{equation}\label{eq2} \chi_1(\mathsf{G})\leq\alpha_1(\mathsf{G})\leq\psi_1(\mathsf{G}).\end{equation}

Consider the cycle $C_n$ of length $n \geq 3$. In this case
$\chi_1(C_n)$ is equal to $2$ if $n$ is even, and is equal to $3$ if
$n$ is odd. On the other hand, it is not hard to see that $\chi_g(C_n)
= n-1$ if $n$ is even and $\chi_g(C_n)= n$ if $n$ is odd.  However,
$\alpha(C_n)= \alpha_1(C_n) = \alpha_g(C_n) = \max \{k \colon k
\lfloor\frac{k}{2}\rfloor \leq n \} - s(n)$, where $s(n)$ is the
number of positive integer solutions to $n=x^2+x+1$. Also, $\psi (C_n)
= \psi_1(C_n) = \psi_g(C_n) = \max \{k\colon k
\lfloor\frac{k}{2}\rfloor \leq n \}$. These results can be found
in~\cite{MR2450569, MR0441778, MR1811221}.

It is known that if $G$ is a planar graph then there always exists a
geometric embedding $j$, where no two edges of $j(G)$ intersect,
except possibly in a common endpoint \cite{MR0026311}. Therefore,
$\psi_1(G) = \psi_1(j(G))=\psi_g(G)$ and $\alpha_1(G) = \alpha_1(j(G))
= \alpha_g(G)$. However, $\chi_1(G)=\chi_1(j(G)) \leq \chi_g(G)$ (for
instance, and as we mentioned before, $\chi_1(C_4) = 2$ and $\chi_g(C_4)=3$).

The chromatic index of a geometric graph $\mathsf{G}$ has been studied
before. Let $l$ be a positive integer and $I(S)$ the graph in which
one vertex corresponds to one subset of $S$ of size $l$, and one edge
corresponds to two vertices of $\mathsf{G}$ whose respective convex
hulls intersect. This graph was defined in \cite{MR2155418}, where the
authors study its chromatic number for the case when $l=2$. If we
denote by $\mathsf{K}_n$ the complete geometric graph with vertex set
$S$, then for the case $l=2$, $\chi(I(S)) = \chi_1( \mathsf{K}_n)$. In
the same paper the authors define and study the number
$i(n)=\max\{\chi(I(S)) \colon S \subset \mathbb{E}^2 \text{ in general
  position, } |S| = n\}$. Note that for the case $l=2$ it happens that
$i(n) = \chi_g(K_n)$. Recall that by $K_n$ we denote the complete
graph on $n$ vertices. The following theorem appears in
\cite{MR2155418}.

\begin{theorem}
  For each $n \geq 3$: i) If the vertices of $\mathsf{K}_n$ are in
  convex position then $\chi_1( \mathsf{K}_n) = n$, ii) $n \leq
  \chi_g(K_n) \leq c n^{3/2}$for some constant $c > 0$.
\end{theorem}

In this paper we prove:

\begin{theorem}\label{thm:big}
i) For each $n \neq 4$, if the vertices of $\mathsf{K}_n$ are in convex position then
\[\alpha_1( \mathsf{K}_n) = \psi_1(\mathsf{K}_n) = \lfloor
\tfrac{n^2+n}{4} \rfloor,\]
ii) For each $n > 18$, $0.0710n^2-\Theta (n) \leq \psi_g(K_n) \leq 0.1781n^{2}+\Theta(n^{\frac{3}{2}}) $.
\end{theorem}

\section{Points in convex position}\label{sec:convex}

In this section we prove Claim i) of Theorem~\ref{thm:big}. In Subsection~\ref{sec:upper} we present an upper bound for $\psi_1(\mathsf{G})$ for any geometric graph $\mathsf{G}$; and then in Subsection~\ref{sec:lower} we exclusively work with point sets in convex position, and derive a tight lower bound for  $\alpha_1(\mathsf{K}_n)$.

\subsection{Upper bound: $\psi_1(\mathsf{G}) \leq \lfloor \tfrac{n^2+n}{4} \rfloor$}\label{sec:upper}

The following theorem was shown in \cite{MR0015796}.

\begin{theorem}\label{teo:erdos}
Any geometric graph with $n$ vertices and $n+1$ edges, contains two disjoint edges.
\end{theorem}

Using this theorem we obtain the following result, where the order of
a graph denotes the number of its vertices.

\begin{cor}\label{cor:size-one}
  Let $\mathsf{G}$ be a geometric graph of order $n$. There are at
  most $n$ chromatic classes of size one in any complete coloring of
  $\mathsf{G}$.
\end{cor}

This corollary immediately implies an upper bound on
$\psi_1(\mathsf{G})$.

\begin{theorem}\label{thm:upper}
  Let $\mathsf{G}$ be a geometric graph of order $n$. The
  pseudoachromatic index $\psi_1(\mathsf{G})$ of $\mathsf{G}$ is at
  most $\lfloor \frac{n^2+n}{4} \rfloor$.

\begin{proof}
  We proceed by contradiction. Assume there exists a geometric graph
  $\mathsf{G}$ for which a complete coloring using $ \lfloor
  \frac{n^2+n}{4} \rfloor + 1$ colors exist. This coloring must have
  at most $\binom{n}{2} - \left( \lfloor \frac{n^2+n}{4} \rfloor + 1
  \right)$ chromatic classes of cardinality larger than one. Thus,
  there are at least $ \lfloor \frac{n^2+n}{4} \rfloor + 1 - \left(
    \binom{n}{2} - \lfloor \frac{n^2+n}{4} \rfloor -1 \right)$
  chromatic classes of size one, that is:

\begin{equation}
    1 - \left(  \binom{n+1}{2}  -2 \left \lfloor \frac{\binom{n+1}{2}}{2}  \right \rfloor  \right)  +n + 1   =
    \begin{cases}
        n + 1 & \text{if  } \binom{n+1}{2} \text{ is odd},\\
        n + 2 & \text{if } \binom{n+1}{2} \text{ is even}.
    \end{cases}
\end{equation}

This contradicts Corollary~\ref{cor:size-one} and therefore the
theorem follows.
\end{proof}
\end{theorem}

\subsection{Tight lower bound: $\alpha_1(\mathsf{G}) \geq \lfloor \frac{n^2+n}{4} \rfloor$}\label{sec:lower}

In this subsection we prove that the bound presented in
Theorem~\ref{thm:upper} is tight. To derive
the lower bound we use a complete geometric graph induced by a set of points
in convex position. We call this type of graph a \emph{complete convex
  geometric graph}. The crossing pattern of the edge set of a complete
convex geometric graph depends only on the number of vertices, and not
on their particular position. Without loss of generality we therefore
assume that the point set of the graph corresponds to the vertices of
a regular polygon. In the remainder of this section we exclusively
work with this type of graphs.

To simplify the proof of the main statement of this section, in the
following we will define different sets of edges and prove some
important properties of these sets.

Let $\mathsf{G}$ be a complete convex geometric graph of order $n$,
and let $\{1, \ldots, n\}$ be the vertices of the graph listed in
clockwise order. For the remainder of this subsection it is important to
bear in mind that all sums are taken modulo $n$; for the sake of
simplicity we will avoid writing this explicitly.  We denote by
$e_{i,j}$ the edge between the vertices $i$ and $j$. We call an edge
$e_{i,j}$ a \emph{halving edge} if in both of the two open semi-planes
defined by the line containing $e_{i,j}$, there are at least $\lfloor
\frac{n-2}{2}\rfloor$ points of $\mathsf{G}$. Using this concept we
obtain the following definition.

\begin{defini}
  Let $i,j,k \in \{1,\ldots,n\}$, such that $e_{i,j}$ and $e_{j+1,k}$
  do not intersect. We call a pair of edges $(e_{i,j}, e_{j+1,k})$ a
  \emph{halving pair of edges} (\emph{halving pair}, for short) if at
  least one of $e_{i,j+1}$, $e_{i,k}$, or $e_{j,k}$ is a halving
  edge. This halving edge is called the \emph{witness} of the halving
  pair.
\end{defini}

See \figurename~\ref{fig:quad} for an example of a halving
pair $(e_{i,j}, e_{j+1,k})$, with $e_{i,k}$ as witness. Note that a halving pair may have more than one witness.

\begin{figure}
  \begin{center}
    \includegraphics{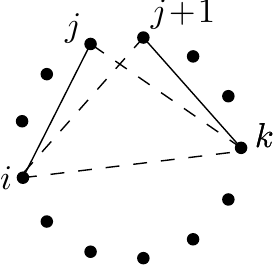}
    \caption{Example for $n=13$. Edges of the halving pair $(e_{i,j},
      e_{{j+1},k})$ are shown solid, dashed edges represent the
      possible halving edges, of which $e_{i,k}$ is a halving edge
      (the witness) in the shown example. The convex hull of a pair
      $(e_{i,j}, e_{{j+1},k})$ is always a quadrilateral.}
    \label{fig:quad}
  \end{center}
\end{figure}

We say that an edge $e$ intersects a pair of edges $(f,g)$ if $e$
intersects at least one of $f$ or $g$. We say that two pairs of edges
intersect if there is an edge in the first pair which intersects the
second pair.

\begin{lemma}\label{lem:halving_intersect}
  Let $\mathsf{G}$ be a complete convex geometric graph of order
  $n$. i) Each two halving edges intersect. ii) Any halving edge
  intersects any halving pair of edges. iii) Any two halving pairs
  intersect.
\end{lemma}

\begin{proof}
  To prove Claim~i) assume that there are two halving edges which do
  not intersect. These edges divide the set of vertices of
  $\mathsf{G}$ into two disjoint sets of size at least $\lfloor
  \frac{n-2}{2} \rfloor$ and one set of size at least $4$ (the
  vertices of the two halving edges). Then, the total number of
  vertices is:

\begin{equation}
    2 \left \lfloor \frac{n-2}{2} \right \rfloor + 4 =
    \begin{cases}
        n + 1 & \text{if n is odd} \\
        n + 2 & \text{if n is even}
    \end{cases}
\end{equation}

This is a contradiction, which proves Claim~i).

To prove Claims~ii) and~iii) observe that the convex hull of each
halving pair $(e_{i,j}, e_{j+1,k})$ defines a quadrilateral
$(i,j,j+1,k)$, see Figure~\ref{fig:quad}. The halving edge witnessing
the halving pair is contained in the corresponding convex hull: it is
either the edge $e_{i,k}$, or one of the diagonals of the
quadrilateral.

It is easy to see, that if either $e_{i,k}$ or one of the diagonals
is intersected by an edge $f$, then $f$ also intersects at least one
edge of the pair $(e_{i,j}, e_{j+1,k})$.

Using this observation we prove the remaining two cases by
contradiction: Assume there exists a halving edge and a halving pair
which do not intersect, or two halving pairs which do not intersect.
Then their corresponding halving edges (witnesses) do not intersect
either, because they are contained in the quadrilaterals. This
contradicts Claim~i), and thus proves Claim~ii) and~iii).
\end{proof}

\begin{defini}
  Let $\mathsf{G}$ be a complete convex geometric graph of even order
  $n$. We call an edge $e_{i,j}$ an \emph{almost-halving edge} if
  $e_{i,j+1}$ is a halving edge.
\end{defini}

Please observe that this definition and the following lemma are only
stated (and valid) for even $n$; therefore if $e_{i,j}$ is an almost-halving edge, $e_{i-1,j}$ is a halving edge.

\begin{lemma}\label{lem:almost_intersect}
  Let $\mathsf{G}$ be a complete convex geometric graph of even order
  $n$. Let $f$ be an almost-halving edge, $e$ a halving edge, and $E$ a
  halving pair. i) $f$ and $e$ intersect, ii) $f$ and $E$ intersect.
\end{lemma}

\begin{proof}
  We prove Claim~i) by contradiction. If $e$ and $f$ do not intersect,
  then they divide the set of vertices of $\mathsf{G}$ into three
  sets: one of size at least $\tfrac{n-2}{2}$, one of size at least
  $\frac{n-2}{2}-1$, and one of size at least $4$. In total the number
  of vertices is (at least):

  \begin{equation}
    2 \left (\frac{n-2}{2} \right )-1 + 4 = n + 1
  \end{equation}

  This is a contradiction, which proves Claim~i). To prove Claim~ii) we
  use Claim~i): the halving edge witnessing $E$ must intersect $f$. On
  the other hand such a halving edge is inside the convex hull of $E$,
  see Figure~\ref{fig:quad}. From these two observations it follows
  that $E$ and $f$ intersect.
\end{proof}

We need two more concepts from the literature. A \emph{straight-line
  thrackle} \cite{MR1476318} of $\mathsf{G}$ is a subset of edges of
$\mathsf{G}$ with the property that any two distinct edges intersect
(they have a common endpoint or they cross). A straight-line
  thrackle is \emph{maximal} if it is not a proper subset of any other thrackle.
Theorem~\ref{teo:erdos}
implies that the size of any straight-line thrackle of $\mathsf{G}$ is
at most $n$. In the following we always refer to a straight-line
thrackle as thrackle, since we are only working with geometric
embeddings of graphs.

Given a set $J \subseteq \{1, \ldots, \lfloor \frac{n}{2}\rfloor\}$, a
\emph{circulant graph} 
$C_n(J)$ of $G$ is defined as the
graph with vertex set equal to $V(G)$\footnote{This definition is different from that usually given, in which the vertex set of the graph is $\mathbb{Z}_n$. However, as it is not hard to see that this definition is equivalent to the usual one, and to keep our arguments as simple as possible, we opt for this choice.}
and $E(C_n(J)) =
\{e_{i,j} \in E(G) \colon j-i \equiv k \mod n, \text{ or }
j-i \equiv -k \mod n\mbox{, } k \in J\}$.
In this paper we use this concept for geometric graphs, in the natural way;
see \figurename~\ref{fig:circulant-third}~(left) for an example of a
circulant geometric graph $C_n(J)$ with
$J=\{\left\lfloor\frac{n}{2}\right\rfloor-1\}$ and $n=13$.

  \begin{figure}[tb]
    \begin{center}
      \includegraphics{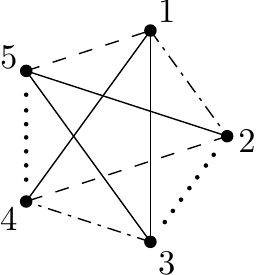}
      \caption{Proof of Theorem~\ref{lowbound} for $n=5$: $\alpha_1(\mathsf{K}_5) = 7$.
      Edges $e_{1,3}, e_{3,5}, e_{4,1}, e_{5,2}$ (solid) are colored with colors $1$ to $4$, respectively.
      Edges $e_{1,2}, e_{3,4}$ are colored with color $5$; and edges $e_{2,3}, e_{4,5}$ are colored with color $6$.
      Finally, edges $e_{2,4},e_{1,5}$ are colored with color $7$.
      Each pair of chromatic classes intersect, and each pair of edges of the same color are disjoint.}
      \label{fig:k5}
    \end{center}
  \end{figure}

\begin{figure}[htb]
  \begin{center}
    \includegraphics{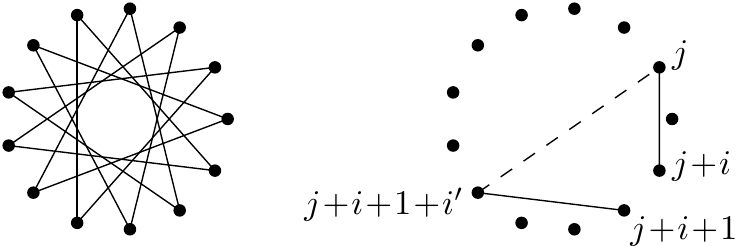}
    \caption{Examples for $n=13$. Left: A circulant graph
      $C_n\left(\left\{ \left\lfloor\frac{n}{2}\right\rfloor-1
        \right\}\right)$. Right: A pair of edges (solid)
      with same color from $E(C_n(\{i,i'\}))$, with $i=2$ and some
      fixed $j$. The witness of the halving pair is shown dashed.}
    \label{fig:circulant-third}
  \end{center}
\end{figure}

The following theorem provides the lower bound on the achromatic index.

\begin{theorem}\label{lowbound}
  Let $\mathsf{G}$ be a complete convex geometric graph of order
  $n \neq 4$. The achromatic index of $\mathsf{G}$ satisfies the following
  bound: \[\alpha_1(\mathsf{G}) \geq \lfloor \tfrac{n^2+n}{4}
  \rfloor.\]
\end{theorem}

\begin{proof}

The theorem follows easily for $n \leq 3$; we prove the case $n=5$ in Figure~\ref{fig:k5}.
For $n > 5$, consider the following partition of the set of edges of $\mathsf{G}$:
\begin{equation}\label{partition}
E(\mathsf{G}) =
E(C_{n}(\{\lfloor\tfrac{n}{2}\rfloor\}))\,
\bigcup\,E(C_{n}(\{\lfloor\tfrac{n}{2}\rfloor-1\}))\,
\bigcup_{i\in I}\,E(C_{n}(\{i,\lfloor\tfrac{n}{2}\rfloor-1-i\}))
\end{equation}
where $I=\{1,\ldots,\lfloor\frac{\lfloor\frac{n}{2}\rfloor-1}{2}\rfloor\}$.

Observe that the first term is a circulant graph of halving edges and thus, by
Lemma~\ref{lem:halving_intersect}, its set of edges defines a
thrackle. This thrackle is maximal (containing $n$ edges) if $n$ is
odd but it is not maximal (containing only $\frac{n}{2}$ edges) if $n$
is even.

Note further, that for fixed $i$ the third term is either the union of two circulant graphs
of size $n$, or one circulant graph of size $n$ (only in the case when $i =
\lfloor\tfrac{n}{2}\rfloor-1-i$).

If $n$ is odd, then the edge set of $\mathsf{G}$ is partitioned into
$\frac{n-1}{2}$ circulant graphs, each of them of size
$n$. If $n$ is even, then the edge set of $\mathsf{G}$ is partitioned into $\frac{n}{2}-1$ circulant graphs, each of
them of size $n$, plus one circulant graph of size $\frac{n}{2}$. Using partition~\ref{partition}, we
give a coloring on the edges of $\mathsf{G}$, and prove that this coloring is
proper and complete.

We start by coloring all circulant graphs in the third term of the
partition, except for $i=\lfloor\frac{\lfloor\frac{n}{2}\rfloor-1}{2}\rfloor$.

In the following we set $i' = \lfloor \frac{n}{2}\rfloor -1-i$ and
therefore refer to $C_{n}(\{i,\lfloor\tfrac{n}{2}\rfloor-1-i\})$ as $C_n(
\{i,i'\})$. For every $i \in I \setminus \left
  \{\lfloor\frac{\lfloor\frac{n}{2}\rfloor-1}{2}\rfloor \right \}$ we
assign colors to $C_n(\{i,i'\})$ using the following function.

\[f_{i}\colon E(C_n(\{i,i'\}))  \longrightarrow  \{(i-1)n+1, \ldots,(i-1)n+n\} \text{ such that: }\]
\begin{eqnarray*}
  e_{j,j+i} & \mapsto & (i-1)n+j \text{, and} \\
  e_{j+i+1, j+i+1+i'} & \mapsto & (i-1)n+j.
\end{eqnarray*}
for $j \in \{1,\dots,n\}$. See
\figurename~\ref{fig:circulant-third}~(right) for an example with $i=2$.

The first rule colors the edges of $C_{n}(\{i\})$, while the second
rule colors the edges of $C_{n}(\{i'\})$. For fixed $i$ and $j$ both
rules assign the same color. Therefore, the chromatic classes are
pairs of edges, one edge ($e_{j,j+i}$) from $C_{n}(\{i\})$ and one
edge ($e_{j+i+1,j+i+1+i'}$) from $C_{n}(\{i'\})$. Observe, that all
these pairs are halving pairs $(e_{j,j+i},e_{j+i+1,j+i+1+i'})$ of
$\mathsf{G}$, because the edge
$e_{j,j+i+1+i'}=e_{j,j+\lfloor\frac{n}{2}\rfloor}$ is halving.

Hence, the partial coloring so far is complete (by
Lemma~\ref{lem:halving_intersect}) and proper (because the two edges
in each color class do not intersect).

The number of colors we have used so far is $N_{1}=n \left(
  \lfloor\tfrac{\left \lfloor \frac{n}{2} \right
    \rfloor-1}{2}\rfloor-1 \right)$.

So far, a subset of edges of the third term of the partition~\ref{partition} is colored. This leaves the following
parts uncolored:

$$E(C_{n}(\{\lfloor\tfrac{n}{2}\rfloor\}))\,
\bigcup\, E(C_{n}(\{\lfloor\tfrac{n}{2}\rfloor-1\}))\,
\bigcup\, E(C_{n}(\{i,i' \}))$$
where $i=\lfloor\frac{\lfloor\frac{n}{2}\rfloor-1}{2}\rfloor$ and $i'=\lfloor\tfrac{n}{2}\rfloor-1-i$.

These remaining circulant graphs differ for $n$ even or
$n$ odd. Further, the two cases $i = i'$ and $i\neq i'$ need to be
distinguished (for the remainder of the third term). This basically
results in the four cases $n \equiv x \mod 4$, for $x\in\{0,1,2,3\}$.

In a nutshell, to color the remaining edges, first the thrackle,
$E(C_{n}(\{\lfloor\tfrac{n}{2}\rfloor\}))$, will be colored (if $n$ is
even together with one half of
$E(C_{n}(\{\lfloor\tfrac{n}{2}\rfloor-1\}))$).  Then (the remaining
half of) the circulant graph $C_{n}(\{\lfloor\tfrac{n}{2}\rfloor-1\})$
together with $C_{n}(\{i,i' \})$
($i=\lfloor\frac{\lfloor\frac{n}{2}\rfloor-1}{2}\rfloor$ and
$i'=\lfloor\tfrac{n}{2}\rfloor-1-i$) is colored. In each step we will
prove, that the (partial) coloring is proper and complete.

\begin{enumerate}
\item Case $n > 5$ is odd. To color the maximal thrackle, $E(C_n(\{ \lfloor \frac{n}{2} \rfloor \}))$, we assign colors
to its edges using the function
\[f\colon E(C_n(\{\lfloor \tfrac{n}{2} \rfloor\}))\longrightarrow \{N_1+1,\dots,N_1+n\} \text{ such that: }\]
\[e_{j, j+\lfloor \frac{n}{2} \rfloor} \mapsto N_1+j,\] for each $j
\in \{1,\dots,n\}$. Observe that $E(C_n(\{ \lfloor \frac{n}{2} \rfloor
\}))$ is a set of $n$ halving edges. See
\figurename~\ref{fig:odd-thrackle}~(left) for an example of such a
thrackle.

\begin{figure}[htb]
  \begin{center}
    \includegraphics{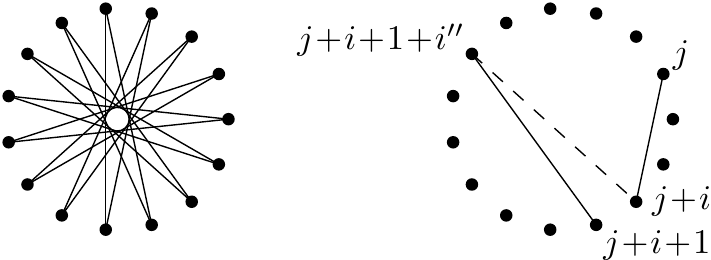}
    \caption{Examples for $n=15$. Left: A circulant graph
      $C_n\left(\left\{ \left\lfloor\frac{n}{2}\right\rfloor
        \right\}\right)$ of halving edges if $n$ is odd. Right:
      Halving pair (solid) with color $N_2+j$ from
      $E(C_n(\{i,i''\}))$, with $n \equiv 3 \mod 4$ and some fixed
      $j$. The witness of the halving pair is shown dashed.}
    \label{fig:odd-thrackle}
  \end{center}
\end{figure}

The coloring so far is proper, because each new chromatic class has
size one. Further, each chromatic class so far consists of either a
halving edge or a halving pair. Hence, by
Lemma~\ref{lem:halving_intersect}, the coloring is also complete.

It is easy to see that we are using $N_2= N_1 + n =
n\lfloor\tfrac{\lfloor \frac{n}{2} \rfloor-1}{2}\rfloor$ colors so far.

The remaining uncolored edges are:
$$E(C_{n}(\{\lfloor\tfrac{n}{2}\rfloor-1\}))\,
\bigcup\, E(C_{n}(\{i,i' \}))$$
where $i=\lfloor\frac{\lfloor\frac{n}{2}\rfloor-1}{2}\rfloor$ and $i'=\lfloor\tfrac{n}{2}\rfloor-1-i$.
These two circulant graphs will be colored together. Let $i''=
\lfloor\frac{n}{2}\rfloor-1$. As $n$ is odd, $C_n(\{i''\})$ consists
of $n$ edges. The size of $E(C_{n}(\{i,i' \}))$ depends on the two
cases $i=i'$ and $i\neq i'$.

\begin{enumerate}
\item $i=i'$: As $n$ is odd, $n \equiv 3 \mod 4$. The circulant graph
  $C_n(\{i,i'\}) = C_n(\{i\})$ is of size $n$. Thus, $2n$ edges remain uncolored.

  We assign $n$ colors to the $2n$ edges of $C_{n}(\{i,i''\})$ as
  follows:

\[f_{i}\colon E(C_n(\{i,i''\}))\longrightarrow \{N_{2}+1,\ldots,N_{2}+n\}\text{, such that}\]
\begin{eqnarray*}
e_{j,j+i} & \mapsto & N_{2}+j \text{,} \\
e_{j+i+1,j+i+1+i''} & \mapsto & N_{2}+j
\end{eqnarray*}
for $j \in \{1,\dots,n\}$.

Each new chromatic class consists of a pair
$(e_{j,j+i},e_{j+i+1,j+i+1+i''})$ of edges. See
\figurename~\ref{fig:odd-thrackle}~(right) for an example of such a
pair. Because the edge $e_{j+i,j+i+1+i''}=e_{j+i,j+i+\lfloor
  \frac{n}{2} \rfloor }$ is a halving edge, the pair
$(e_{j,j+i},e_{j+i+1,j+i+1+i''})$ is a halving pair.
Therefore, all edges are colored and each chromatic class consists of
either a halving edge or a halving pair. By
Lemma~\ref{lem:halving_intersect} the coloring is complete and proper
(as the edges of halving pairs are disjoint).

The total number of colors used is $N_{3}=N_2 + n = n (
\lfloor\tfrac{\lfloor \frac{n}{2} \rfloor-1}{2}\rfloor+1 )$, that is
$N_3=\left\lfloor \frac{n^2 + n}{4}\right\rfloor$ colors, as $n \equiv
3 \mod 4$ in this case.

\item $i \not =i'$: As $n$ is odd, $n \equiv 1 \mod 4$. The circulant
  graph $C_n(\{i,i'\})$ is of size $2n$. Thus, $3n$ edges
  remain uncolored.
  We assign $n$ colors to the $2n$ edges of $C_{n}(\{i,i'\})$ and
  $\left\lfloor\frac{n}{2}\right\rfloor$ colors to the $n$ edges of
  $C_{n}(\{i''\})$ as follows:

\[f_{i}\colon E(C_n(\{i,i',i''\}))\longrightarrow \{N_{2}+1,\ldots,N_{2}+n+\lfloor \tfrac{n}{2} \rfloor\}\text{, such that}\]
\begin{eqnarray*}
e_{j,j+i} & \mapsto & N_{2}+j \text{,}\\
e_{j+i+1,j+i+1+i'} & \mapsto & N_{2}+j,
\end{eqnarray*}
$\text{for } j \in \{1, \ldots,n \}$, and
\begin{eqnarray*}
e_{j,j+i''} & \mapsto & N_{2}+n+j \text{,} \\
e_{j+i''+1,j+i''+1+i''} & \mapsto & N_{2}+n+j,
\end{eqnarray*}
$\text{for } j \in \{ 1, \ldots, \lfloor \tfrac{n}{2} \rfloor \}$. See
\figurename~\ref{fig:odd-b}~(left and~middle) for examples.

\begin{figure}[htb]
  \begin{center}
    \includegraphics{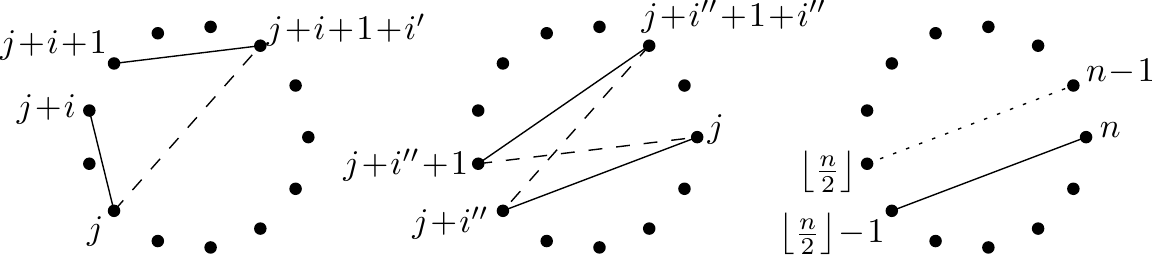}
    \caption{Examples with $n=13$, for $n$ is odd and $i \not =i'$: $n
      \equiv 1 \mod 4$. Left: Halving pair with color $N_2+j$ from
      $E(C_n(\{i,i'\}))$. Middle: Halving pair with color $N_2+n+j$
      from $E(C_n(\{i''\}))$. Both for fixed $j$. Halving pairs are
      shown solid, witnesses of the halving pairs are shown
      dashed. Right: The single remaining edge $e_{n,\lfloor
        \frac{n}{2} \rfloor-1}$ (solid) is combined with the halving
      edge
      $e_{\lfloor\frac{n}{2}\rfloor,n-1}$
      (dotted), colored with color $N_1+\lfloor\frac{n}{2}\rfloor$.}
    \label{fig:odd-b}
  \end{center}
\end{figure}

Each new chromatic class consists of a pair of edges. These pairs are
either $(e_{j,j+i},e_{j+i+1,j+i+1+i'})$ or
$(e_{j,j+i''},e_{j+i''+1,j+i''+1+i''})$ combined from the edges of
$C_{n}(\{i,i'\})$ or $C_{n}(\{i''\})$, respectively.
The pair $(e_{j,j+i},e_{j+i+1,j+i+1+i'})$ is a halving pair with the
halving edge $e_{j,j+i+1+i'}=e_{j,j+\lfloor\frac{n}{2}\rfloor}$ as
witness, and $(e_{j,j+i''},e_{j+i''+1,j+i''+1+i''})$ is a halving pair with the
halving edges $e_{j,j+i''+1}=e_{j,j+\lfloor\frac{n}{2}\rfloor}$ and
$e_{j+i'',j+i''+1+i''}=e_{j+i'',j+i''+\lfloor\frac{n}{2}\rfloor}$ as witnesses.
Each chromatic class so far consists of either a halving edge or a
halving pair. Hence, the coloring is complete (by
Lemma~\ref{lem:halving_intersect}) and proper (as the edges of halving
pairs are disjoint).

Note, that a single edge, $e_{n,\lfloor \tfrac{n}{2} \rfloor-1}$ of $C_{n}(\{i''\})$,
remains uncolored. We add this edge to the chromatic class (with color
$N_1+\lfloor\frac{n}{2}\rfloor$) containing the halving edge
$e_{\lfloor\frac{n}{2}\rfloor,n-1} =
e_{\lfloor\frac{n}{2}\rfloor,\lfloor\frac{n}{2}\rfloor+\lfloor\frac{n}{2}\rfloor}$. See
\figurename~\ref{fig:odd-b}~(right). Observe, that $e_{n,\lfloor
  \frac{n}{2} \rfloor-1}$ and $e_{\lfloor\frac{n}{2}\rfloor,n-1}$ are
disjoint, thus the coloring remains proper. Further, adding an edge to
an existing chromatic class of a complete coloring, maintains the
completeness of the coloring.

As all edges are colored, the total number of colors used is
$N_{3}=N_2+n+\lfloor \frac{n}{2} \rfloor = n( \lfloor\tfrac{\lfloor
  \frac{n}{2} \rfloor-1}{2}\rfloor+1 )+\lfloor \tfrac{n}{2} \rfloor$,
that is $N_3=\left\lfloor \frac{n^2 + n}{4}\right\rfloor$, as $n
\equiv 1 \mod 4$ in this case.
\end{enumerate}


\item Case $n > 5$ is even. Recall that only
  $N_1$ chromatic classes exist so far, each containing a halving pair
  of edges. The thrackle
  $E(C_{n}(\{\lfloor\frac{n}{2}\rfloor\}))=E(C_{n}(\{\frac{n}{2}\}))$
  is not maximal in this case. See
  \figurename~\ref{fig:even-thrackle}~(left). To get a maximal
  thrackle we add half the edges of $C_n(\{\frac{n}{2}-1\})$ to
  $C_{n}(\{\frac{n}{2}\})$.
  Note that $E(C_n(\{\frac{n}{2}-1\}))$ is the set of
  almost-halving edges in the case of $n$ is even.

  \begin{figure}[htb]
    \begin{center}
      \includegraphics{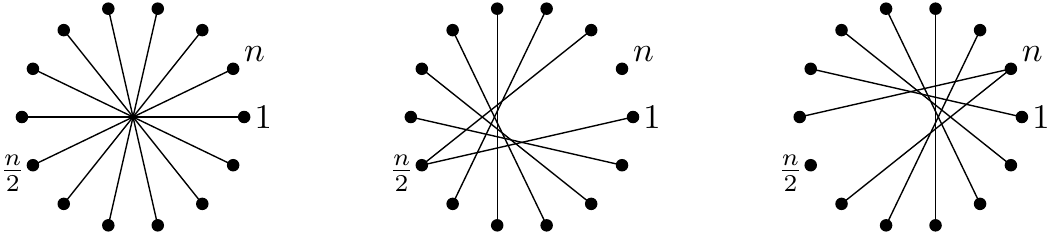}
      \caption{Examples with $n=14$, for the case when $n$ is
        even. Left: The thrackle,
        $E(C_{n}(\frac{n}{2}\}))$, of the $\frac{n}{2}$
        halving edges. Middle: The thrackle, $E(C'_n(\{ \frac{n}{2}
        -1\}))$, of the first $\frac{n}{2}$ almost-halving edges of
        $E(C_n(\{ \frac{n}{2} -1\}))$. Right: The thrackle, $E(C''_n(\{
        \frac{n}{2} -1\}))$, of the second $\frac{n}{2}$
        almost-halving edges of $E(C_n(\{ \frac{n}{2} -1\}))$.}
      \label{fig:even-thrackle}
    \end{center}
  \end{figure}

  Let the thrackle $E(C'_n(\{ \frac{n}{2}
  -1\}))=\{e_{1,\frac{n}{2}},\dots,e_{\frac{n}{2},n-1} \}$ and the thrackle
  $E(C''_n(\{ \frac{n}{2}
  -1\}))=\{e_{\frac{n}{2}+1,n},\dots,e_{n,\frac{n}{2}-1} \}$ define
  the two halves of $C_n(\{\frac{n}{2}-1\})$ with $\frac{n}{2}$
  almost-halving edges each. See
  \figurename~\ref{fig:even-thrackle}~(middle and right).
  It is easy to see that $E(C'_n(\{ \frac{n}{2} -1\}))$ is a thrackle
  (all its edges intersect each other). Further, by
  Lemma~\ref{lem:almost_intersect}, each almost-halving edge
  intersects each halving edge. Thus, $E(C_n(\{ \frac{n}{2} \}) \cup
  C'_n(\{ \frac{n}{2}-1\}))$ is a maximal thrackle of size $n$. The
  following function assigns one color to each edge of this maximal
  thrackle.

\[f\colon E(C_n(\{\frac{n}{2}\}) \cup C'_n(\{ \frac{n}{2}-1 \}))\longrightarrow \{N_1+1,\dots,N_1+n\}\text{, such that}\]
\begin{eqnarray*}
e_{j,j+\frac{n}{2}} & \mapsto & N_{1}+j \text{,} \\
e_{j,j+\frac{n}{2}-1} & \mapsto & N_{1}+\frac{n}{2}+j
\end{eqnarray*}
for each $j \in \{1,\dots,\tfrac{n}{2}\}$.

The coloring so far is proper, because each new chromatic class has
size one. Further, each chromatic class consists of either a
halving edge, a halving pair, or an almost-halving edge.
The almost-halving edges used so far form a thrackle and thus,
intersect each other. Hence, by Lemmas~\ref{lem:halving_intersect}
and~\ref{lem:almost_intersect}, the coloring is also complete. It is
easy to see that we are using $N_2 = N_1 + n =
n\lfloor\tfrac{n-2}{4}\rfloor $ colors so far.

The remaining uncolored edges are
$$E(C_{n}''(\{\lfloor\tfrac{n}{2}\rfloor-1\}))\,
\bigcup\, E(C_{n}(\{i,i' \}))$$
where $i=\lfloor\frac{n-2}{4}\rfloor$ and $i'=\frac{n}{2}-1-i$ (as $n$
is even).
These two circulant graphs will be colored together. For brevity, let
$i''= \frac{n}{2}-1$ and $C_n''(\{i''\})$ be the set of the remaining
$\frac{n}{2}$ almost-halving edges. The size of $E(C_{n}(\{i,i' \}))$
depends on the two cases $i=i'$ and $i\neq i'$.

\begin{enumerate}

\item $i = i'$: As $n$ is even, $n \equiv 2 \mod 4$. The circulant
  graph $C_n(\{i,i'\}) = C_n(\{i\})$ is of size $n$. Thus,
  $n+\frac{n}{2}$ edges remain uncolored. We assign $\frac{n}{2} + \lfloor\frac{n}{4}\rfloor$ colors to the $n+\frac{n}{2}$ edges of
  $C_{n}(\{i\}) \cup C''_{n}(\{i''\})$ as follows:
\[f_{i}\colon E(C_n(\{i\})\cup C''_n(\{ i'' \}))\longrightarrow \{N_{2}+1,\ldots,N_{2}+\tfrac{n}{2}+\lfloor\tfrac{n}{4}\rfloor\}\text{, such that}\]
\begin{eqnarray*}
e_{\tfrac{n}{2}+j,\tfrac{n}{2}+j+i''} & \mapsto & N_{2}+j \text{,} \\
e_{\tfrac{n}{2}+j+i''+1, \tfrac{n}{2}+j+i''+1+i} & \mapsto & N_{2}+j
\end{eqnarray*}
for $j \in \{1,\dots,\tfrac{n}{2}\}$,  and
\begin{eqnarray*}
e_{\tfrac{n}{2}+j,\tfrac{n}{2}+j+i} & \mapsto & N_{2}+\tfrac{n}{2}+j \text{,} \\
e_{\tfrac{n}{2}+j+i+1,\tfrac{n}{2}+j+i+1+i} & \mapsto & N_{2}+\tfrac{n}{2}+j
\end{eqnarray*}
for $j \in \{1,\dots,\lfloor\tfrac{n}{4}\rfloor\}$. See
\figurename~\ref{fig:even-a}~(left and~middle) for examples.

  \begin{figure}[tb]
    \begin{center}
      \includegraphics{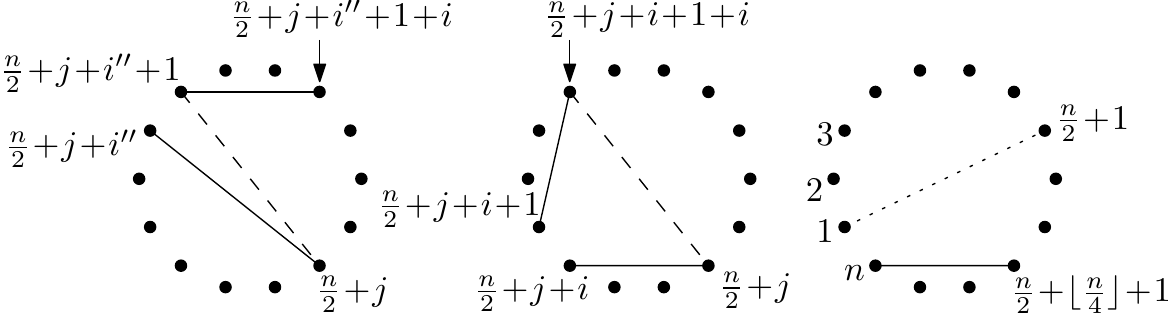}
      \caption{Examples with $n=14$, for the case when $n$ is even and
        $i = i'$: $n \equiv 2 \mod 4$. Left: Halving pair with color
        $N_2+j$. Middle: Halving pair with color
        $N_2+\frac{n}{2}+j$. Both for fixed $j$. Halving pairs are
        shown solid, witnesses of the halving pairs are shown
        dashed. Right: The single remaining edge
        $e_{\frac{n}{2}+\lfloor\frac{n}{4}\rfloor+1,n}$ (solid) is
        combined with the halving edge $(e_{1,\frac{n}{2}+1})$
        (dotted), colored with color~$N_1+1$.}
      \label{fig:even-a}
    \end{center}
  \end{figure}

  Each new chromatic class consists of a halving pair of edges from
  $E(C_n(\{i\})\cup C''_n(\{ i'' \}))$, either
  $(e_{\frac{n}{2}+j,\frac{n}{2}+j+i''},e_{\frac{n}{2}+j+i''+1,
    \frac{n}{2}+j+i''+1+i})$ with the halving edge
  $e_{\frac{n}{2}+j,\frac{n}{2}+j+i''+1}=e_{\frac{n}{2}+j,\frac{n}{2}+j+\frac{n}{2}}$
  as witness, or
  $(e_{\frac{n}{2}+j,\frac{n}{2}+j+i},e_{\frac{n}{2}+j+i+1,\frac{n}{2}+j+i+1+i})$
  with, again, the halving edge
  $e_{\frac{n}{2}+j,\frac{n}{2}+j+i+1+i}=e_{\frac{n}{2}+j,\frac{n}{2}+j+\frac{n}{2}}$
  as witness.
  Each chromatic class so far consists of either a halving edge, a
  halving pair, or one of the $\frac{n}{2}$ almost-halving edges that
  form a thrackle. Hence, the coloring is complete (by
  Lemmas~\ref{lem:halving_intersect} and~\ref{lem:almost_intersect})
  and proper (as the edges of halving pairs are disjoint).

  Note, that a single edge,
  $e_{\frac{n}{2}+\lfloor\frac{n}{4}\rfloor+1,n}$ of $C_{n}(\{i\})$,
  remains uncolored. We add this edge to the chromatic class (with
  color $N_1+1$) containing the halving edge
  $(e_{1,\frac{n}{2}+1})$. See
  \figurename~\ref{fig:even-a}~(right). Observe, that
  $e_{\frac{n}{2}+\lfloor\frac{n}{4}\rfloor+1,n}$ and
  $(e_{1,\frac{n}{2}+1})$ are disjoint. Thus, the coloring
  remains proper. Further, adding an edge to an existing chromatic
  class of a complete coloring, maintains the completeness of the
  coloring.

  As all edges are colored, the total number of colors used is
  $N_{3}=N_2 + \frac{n}{2}+\lfloor\frac{n}{4}\rfloor =
  n\lfloor\frac{n-2}{4}\rfloor +
  \frac{n}{2}+\lfloor\frac{n}{4}\rfloor$, that is $N_3=\left\lfloor
    \frac{n^2 + n}{4}\right\rfloor$, as $n \equiv 2 \mod 4$ in this
  case.

\item $i \neq i'$: As $n$ is even, $n \equiv 0 \mod 4$. The circulant
  graph $C_n(\{i,i'\})$ is of size $2n$. Thus,
  $2n+\frac{n}{2}$ edges remain uncolored. We assign $\frac{n}{2} + 3\frac{n}{4}$ colors to the $2n+\frac{n}{2}$ edges of
  $C_{n}(\{i,i'\}) \cup C''_{n}(\{i''\})$ as follows:
\[f_{i}\colon E(C_n(\{i,i'\})\cup C''_{n}(\{i''\}))\longrightarrow \{N_{2}+1,\ldots,N_{2}+\tfrac{n}{2}+3\tfrac{n}{4}\}\text{, such that}\]
\begin{eqnarray*}
e_{\tfrac{n}{2}+j,\tfrac{n}{2}+j+i''} & \mapsto & N_{2}+j \text{,} \\
e_{\tfrac{n}{2}+j+i''+1,\tfrac{n}{2}+j+i''+1+i} & \mapsto & N_{2}+j \text{,}
\end{eqnarray*}
\begin{eqnarray*}
e_{3\tfrac{n}{4}+j,3\tfrac{n}{4}+j+i''} & \mapsto & N_{2}+\tfrac{n}{4}+j \text{,} \\
e_{3\tfrac{n}{4}+j+i''+1,3\tfrac{n}{4}+j+i''+1+i'} & \mapsto & N_{2}+\tfrac{n}{4}+j \text{,}
\end{eqnarray*}
for each $j \in \{1,\dots,\tfrac{n}{4}\}$,  and
\begin{eqnarray*}
e_{\tfrac{n}{4}+j,\tfrac{n}{4}+j+i} & \mapsto & N_{2}+\tfrac{n}{2}+j \text{,} \\
e_{\tfrac{n}{4}+j+i+1,\tfrac{n}{4}+j+i+1+i'} & \mapsto & N_{2}+\tfrac{n}{2}+j
\end{eqnarray*}
for each $j \in \{1,\dots,3\tfrac{n}{4} \}$.
  See \figurename~\ref{fig:even-b} for an example of these three
  different types of pairs of edges.

  \begin{figure}[tb]
    \begin{center}
      \includegraphics{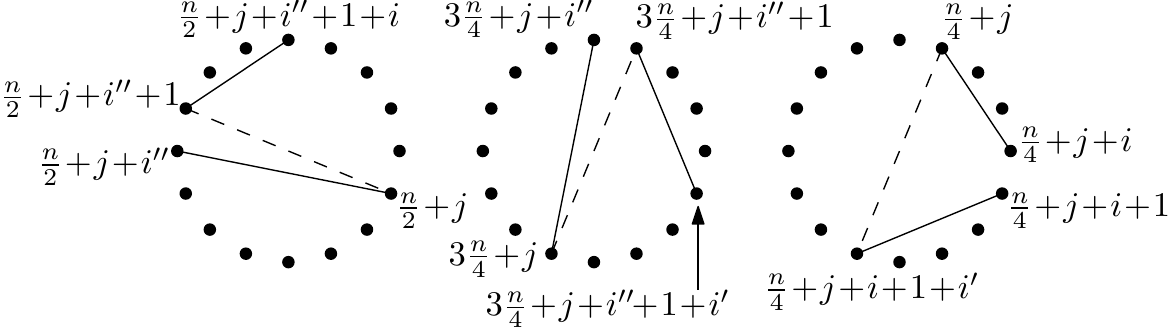}
      \caption{Examples with $n=16$, for the case when $n$ is even and
        $i \neq i'$: $n \equiv 0 \mod 4$. Left: Halving pair with color
        $N_2+j$. Middle: Halving pair with color
        $N_2+\frac{n}{4}+j$. Right: Halving pair with color
        $N_2+\frac{n}{2}+j$. All for fixed $j$. Halving pairs are
        shown solid, witnesses of the halving pairs are shown
        dashed.}
      \label{fig:even-b}
    \end{center}
  \end{figure}

  Each new chromatic class consists of a halving pair of edges from
  $C_n(\{i,i'\})\cup C''_n(\{ i'' \})$, either
  $(e_{\frac{n}{2}+j,\frac{n}{2}+j+i''},e_{\frac{n}{2}+j+i''+1,
    \frac{n}{2}+j+i''+1+i})$ with the halving edge
  $e_{\frac{n}{2}+j,\frac{n}{2}+j+i''+1}=e_{\frac{n}{2}+j,\frac{n}{2}+j+\frac{n}{2}}$
  as witness (\figurename~\ref{fig:even-b}~(left)),
  $(e_{3\frac{n}{4}+j,3\frac{n}{4}+j+i''},e_{3\frac{n}{4}+j+i''+1,3\frac{n}{4}+j+i''+1+i'})$
  with the halving edge
  $e_{3\frac{n}{4}+j,3\frac{n}{4}+j+i''+1}=e_{3\frac{n}{4}+j,3\frac{n}{4}+j+\frac{n}{2}}$
  as witness (\figurename~\ref{fig:even-b}~(middle)), or
  $(e_{\frac{n}{4}+j,\frac{n}{4}+j+i},e_{\frac{n}{4}+j+i+1,\frac{n}{4}+j+i+1+i'})$
  with, again, the halving edge
  $e_{\frac{n}{4}+j,\frac{n}{4}+j+i+1+i'}=e_{\frac{n}{4}+j,\frac{n}{4}+j+\frac{n}{2}}$
  as witness (\figurename~\ref{fig:even-b}~(right)).
  Each chromatic class so far consists of either a halving edge, a
  halving pair, or one of the $\frac{n}{2}$ almost-halving edges that
  form a thrackle. Hence, the coloring is complete (by
  Lemmas~\ref{lem:halving_intersect} and~\ref{lem:almost_intersect})
  and proper (as the edges of halving pairs are disjoint).

  As all edges are colored, the total number of colors used is
  $N_{3}=N_2 + \frac{n}{2}+ 3\frac{n}{4} =
  n\lfloor\frac{n-2}{4}\rfloor + \frac{n}{2}+3\frac{n}{4}$, that is
  $N_3=\left\lfloor \frac{n^2 + n}{4}\right\rfloor$, as $n \equiv 0
  \mod 4$ in this case.

\end{enumerate}
\end{enumerate}
\end{proof}

\begin{proof}[\textbf{Proof of Theorem~\ref{thm:big} i)}]

For $n \neq 4$, by Theorem~\ref{lowbound} we get that $ \left\lfloor
    \frac{n^2+n}{4} \right\rfloor\leq \alpha_1(\mathsf{G})$, and by
  Theorem~\ref{thm:upper} and Equation~\ref{eq2} we conclude that
  $\alpha_1(\mathsf{G}) = \psi_1(\mathsf{G}) = \left\lfloor \frac{n^2
      + n}{4}\right\rfloor$.

\end{proof}

Theorem~\ref{thm:big} i) excludes the case $n=4$, this is because $\mathsf{K}_4$ is the only complete convex geometric graph for which $\alpha_1$ and $\psi_1$ are different. We now prove that $\alpha_1(\mathsf{K}_4) = 4$ and $\psi_1(\mathsf{K}_4)=5$. By Theorem~\ref{thm:upper} we have that $\psi_1(\mathsf{K}_4) \leq 5$, and by Figure~\ref{fig:k4} (left) we can conclude that $\psi_1(\mathsf{K}_4)=5$. Now, by Figure~\ref{fig:k4} (right) we have $\alpha_1(\mathsf{K}_4) \geq 4$. Suppose that $\alpha_1(\mathsf{K}_4) = 5$, then in the coloring of $\mathsf{K}_4$ with $5$ colors, there must be four color classes of size one, and one color class of size two. The class of size two must be composed by two opposite edges of $\mathsf{K}_4$, this implies that the remaining two opposite edges belong to classes of size one; but this is a contradiction because the coloring must be complete. Therefore, it follows that $\alpha_1(\mathsf{K}_4) = 4$.

\begin{figure}[h!]
    \begin{center}
      \includegraphics[scale=1.0]{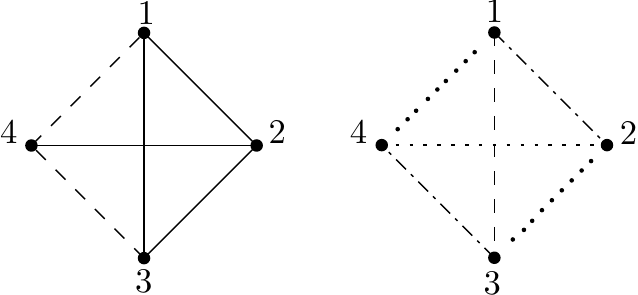}
      \caption{ Left: $\psi_1(\mathsf{K}_4) = 5$. One color class of size two $\{e_{3,4},e_{4,1}\}$; and four color classes of size one $\{e_{1,2}\}$,$\{e_{2,3}\}$,$\{e_{1,3}\}$,$\{e_{2,4}\}$.
      Right: $\alpha_1(\mathsf{K}_4)=4$. Two color classes of size two: $\{e_{1,2},e_{3,4}\}$, $\{e_{2,3}, e_{4,1}\}$;
      and two color classes of size one $\{e_{1,3}\}$, $\{e_{4,2}\}$. }
      \label{fig:k4}
    \end{center}
  \end{figure}

\section{On $\psi_g(K_n)$}\label{sec:general}

In this section we consider point sets in general position in the
plane, and present lower and upper bounds for the geometric
pseudoachromatic index. Recall that the geometric pseudoachromatic index of a graph $G$ is defined as:

$$\psi_g(G)  = \min\{\psi_1(\mathsf{G})\colon \mathsf{G} \textrm{ is a geometric graph of }G\}.$$

\subsection{Upper bound for $\psi_g(K_n)$}

Let $G=(V,E)$ be a graph, and let $\mathsf{G}$ be a geometric representation of $G$. Consider two intersecting edges in $\mathsf{G}$, the intersection might occur either at a common interior point (crossing), or at a common end point (at a vertex).

On one hand, if we consider all edges of $\mathsf{G}$ and denote by $m$ the total number of intersections that occur at vertices, $m$ is precisely the number of edges in the line graph of $G$. That is, if $deg(v)$ is the degree of $v \in V$, then $m=\sum_{v \in V}\binom{deg(v)}{2}.$

On the other hand, the \emph{rectilinear crossing number of $\mathsf{G}$}, denoted by $\overline{cr}(\mathsf{G})$, is defined as the number of edge crossings that occur in $\mathsf{G}$. Given a graph $G$, the \emph{rectilinear crossing number of $G$} is the minimum number of crossings over all possible geometric embeddings of $G$; notationally
$$\overline{cr}(G) = \min\{\overline{cr}(\mathsf{G})\colon \mathsf{G} \textrm{ is a geometric graph of }G\}.$$

It seems natural that there should be a relationship between the
rectilinear crossing number of a graph, and its geometric achromatic
and pseudoachromatic indices. In the following lines we establish bounds for $\psi_g(G)$ as a function of $m$ and $\overline{cr}(G)$.

\begin{lemma} \label{lemma:crossing}
Let $\mathsf{G}$ be a geometric graph of order $n$, denote by $\overline{cr}(\mathsf{G})$ the
number of edge crossings in $\mathsf{G}$, and by $m$ the total number of edge intersections occurring at vertices of $\mathsf{G}$. Then:
$$\psi_1(\mathsf{G}) \leq \left \lfloor \tfrac{1+\sqrt{1+8(m+\overline{cr}(\mathsf{G}))}}{2}\right \rfloor .$$

\begin{proof}
The total number of edge intersections is $m + \overline{cr}(\mathsf{G})$. Then, $m + \overline{cr}(\mathsf{G}) \geq \tbinom{\psi_1(\mathsf{G})}{2}$ so
  that $\psi_1(\mathsf{G})(\psi_1(\mathsf{G})-1) \leq 2(m +
  \overline{cr(\mathsf{G})})$. Solving this inequality we get
  $\psi_1(\mathsf{G}) \leq \left \lfloor
    \tfrac{1+\sqrt{1+8(m+\overline{cr}(\mathsf{G}))}}{2}\right
  \rfloor$.
\end{proof}
\end{lemma}

Using the above Lemma, we can establish the following result.

\begin{theorem} \label{teo:generalbound}
    Let $G = (V,E)$ be a graph of order $n$, with $m=\sum_{v \in V}\binom{deg(v)}{2}$. Denote by $\overline{cr}(G) $ its rectilinear
    crossing number. Then $$\psi_g(G) \leq \left \lfloor \frac{1+\sqrt{1+8(m+\overline{cr}(G)}}{2}\right \rfloor .$$

\begin{proof}
    Let $\mathsf{G}_0$ be a geometric representation of $G$ such that $\overline{cr}(\mathsf{G}_0)=\overline{cr}(G)$; that is, $\mathsf{G_0}$ is a geometric graph of $G$ with minimum number of crossings. As a consequence of Lemma~\ref{lemma:crossing} we have the following:
    \begin{eqnarray*}
        \psi_g(G) & = & \min\{\psi_1(\mathsf{G})\colon \mathsf{G} \textrm{ is a geometric graph of }G\}\leq \psi_1(\mathsf{G_0}) \\
        & \leq &  \left\lfloor \frac{1+\sqrt{1+8(m+\overline{cr}(\mathsf{G}_0))}}{2}\right\rfloor = \left\lfloor \frac{1+\sqrt{1+8(m+\overline{cr}(G))}}{2}\right\rfloor
    \end{eqnarray*}
\end{proof}
\end{theorem}

Establishing bounds for $\overline{cr}(K_n)$ is a well studied problem in the literature. If we use these results, we can give a better bound for $\psi_g(K_n)$; note that in this case $m = n\binom{n-1}{2}$. The following result was shown in \cite{MR2652001}.

\begin{theorem} \label{teo:crossing}
    $\overline{cr}(K_{n})\leq c\binom{n}{4}+\Theta(n^{3})$ for $c=0.380488$.
\end{theorem}

Using this theorem we obtain:

\begin{theorem}
    Let $K_n$ be the complete graph of order $n$. The geometric pseudoachromatic index of $K_n$ has the following upper bound:
    $$\psi_{g}(K_{n})\leq0.1781n^{2}+\Theta(n^{\frac{3}{2}}).$$

\begin{proof}

    Since $m=\underset{v\in V(K_{n})}{\sum}\binom{deg(v)}{2}=n\binom{n-1}{2}$, by Theorem \ref{teo:generalbound} and \ref{teo:crossing},

    \begin{eqnarray*}
        \psi_{g}(K_{n}) & \leq & \frac{1}{2}\sqrt{8\overline{cr}(K_{n}))+\Theta(n^{3})}+\Theta(1) = \frac{1}{2}\sqrt{8\frac{c}{4!}n^{4}+\Theta(n^{3})}+\Theta(1) \\
        & =& \sqrt{\frac{c}{12}}n^{2}+\Theta(n^{\frac{3}{2}}) \leq 0.1781n^{2}+\Theta(n^{\frac{3}{2}}).
    \end{eqnarray*}
\end{proof}
\end{theorem}

\subsection{Lower bound for $\psi_g(K_n)$}

In this section we present a lower bound for $\psi_g(K_n)$. First let us state a result which will be used later; this result was shown in~\cite{MR0188890}.

\begin{theorem}\label{thm:sixpartition}
Let $S$ be a set of $n$ points in general position in the plane.
There are three concurrent lines that divide the plane into six
parts each containing at least $\tfrac{n}{6} - 1$ points of $S$ in its interior.
\end{theorem}

In order to exhibit the desired coloring, first we divide the
plane into seven regions and then use this partition of the plane to construct a partition of
the edges of the complete geometric graph of order $n > 18$. We utilize a specific configuration
$\mathcal{L}$ of lines, defined as follows; see also
Figure~\ref{fig:linestwo} for a drawing of the configuration. Let $S$
be a set of $n=13m+6+r$ points in general position in the plane ($r <
13$). Choose horizontal lines $\ell_1$, $\ell_2$, and $\ell_3$ (listed
top-down) so that when writing $A', B'$ for the set of points between
$\ell_1$ \& $\ell_2$, and $\ell_2$ \& $\ell_3$, respectively, we have $|A'| =
12m+6$ and $|B'|=m+r$. Let $\ell_4$, $\ell_5$, $\ell_6$ be concurrent lines
that divide the set $A'$ into $6$ parts, each containing at least $2m$ points
in its interior; the existence of these lines is guaranteed by Theorem~\ref{thm:sixpartition}. Let $p$ be the point of intersection of the three lines. For each one of the six subsets of points induced by the partition, we take a subset of size exactly $2m$. Let $A,B,C,D,E,F$ be such subsets, listed in clockwise order. Take $G \subseteq B'$ such that $|G| = m$.

\begin{figure}[!htbp]
  \begin{center}
    \includegraphics{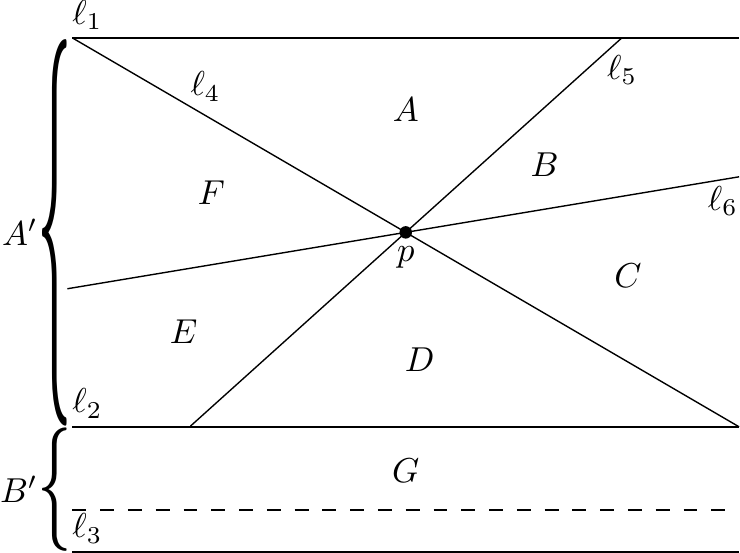}
    \caption{The line configuration $\mathcal{L}$.}
    \label{fig:linestwo}
  \end{center}
\end{figure}

Using these sets of points, first we construct three sets of subgraphs of $\mathsf{K}_n$. Then, we assign one color to each of them and show that such a coloring is complete. Let $A=\{a_1, \ldots, a_{2m}\}, B=\{b_1, \ldots, b_{2m}\}, C=\{c_1, \ldots, c_{2m}\}, D=\{d_1, \ldots, d_{2m}\}, E=\{e_1, \ldots,e_{2m}\}, F=\{f_1, \ldots, f_{2m}\}$; and $G = \{g_1, \ldots, g_m\}$.

For $i,j \in \{1, \ldots, 2m\}$, consider the following sets of subgraphs of $\mathsf{K_n}$:

\begin{itemize}

\item The subgraphs $\mathsf{X}_{i,j}$ with vertex set $\{a_i, b_j, d_i, e_j, g_{\left\lceil \frac{j}{2}\right\rceil }\}$ and edges

\begin{equation*}
    \{a_ib_j, b_jd_i, d_ie_j, e_ja_i \} \cup
    \begin{cases}
        \{a_ig_{ \frac{j}{2}}\} & \text{ if } j \text{ is even } \\
        \{d_ig_{\frac{j+1}{2}}\} &\text{ if } j \text{ is odd }.
    \end{cases}
\end{equation*}

Note that each $\mathsf{X}_{i,j}$ is a quadrilateral plus one edge. We call each quadrilateral $\mathsf{X}'_{i,j} \leq \mathsf{X}_{i,j}$, induced by vertices $a_i,b_j,d_i,e_j$.

\item The subgraphs $\mathsf{Y}_{i,j}$ with vertex set $\{b_i, c_j, e_i, f_j, g_{\left\lceil \frac{j}{2}\right\rceil }\}$ and edges

\begin{equation*}
    \{b_ic_j, c_je_i, e_if_j, f_jb_i \} \cup
    \begin{cases}
        \{b_ig_{ \frac{j}{2} }\} & \text{ if } j \text{ is even } \\
        \{e_ig_{\frac{j+1}{2}}\} &\text{ if } j \text{ is odd }.
    \end{cases}
\end{equation*}

Let $\mathsf{Y}'_{i,j} \leq \mathsf{Y}_{i,j}$ be the quadrilateral induced by vertices $b_i,c_j,e_i,f_j$.

\item The subgraphs $\mathsf{Z}_{i,j}$ with vertex set $\{c_i, d_j, f_i, a_j, g_{\left\lceil \frac{j}{2}\right\rceil }\}$ and edges
\begin{equation*}
    \{c_id_j, d_jf_i, f_ia_j, a_jc_i \} \cup
    \begin{cases}
        \{ c_ig_{ \frac{j}{2} } \} & \text{ if } j \text{ is even } \\
        \{ f_ig_{\frac{j+1}{2}} \} &\text{ if } j \text{ is odd }.
    \end{cases}
\end{equation*}

Let $\mathsf{Z}'_{i,j} \leq \mathsf{Z}_{i,j}$ be the quadrilateral induced by vertices $c_i,d_j,f_i,a_j$.

\end{itemize}

Please note that the set $\{\mathsf{X}_{i,j}, \mathsf{Y}_{i,j}, \mathsf{Z}_{i,j}\}$, with $i,j \in \{1, \ldots, 2m\}$ contains exactly $12m^2$ subgraphs. Also note that each subgraph $\mathsf{X'}_{i,j},\mathsf{Y'}_{i,j}$ and $\mathsf{Z'}_{i,j}$, is a (not necessarily convex) quadrilateral. The following lemma shows that $p$ is inside each of these quadrilaterals.

\begin{lemma}\label{lem:p_inside}

Let $p$ be the point of intersection of the three lines in Theorem~\ref{thm:sixpartition}. Then $p$ is inside each of the polygons induced by the graphs $\mathsf{X'}_{i,j},\mathsf{Y'}_{i,j}$ and $\mathsf{Z'}_{i,j}$, defined above.

\begin{proof}

Consider the polygon $P$ induced by $\mathsf{X}'_{i,j}$ and recall that the vertices of $P$ are $\{a_i,b_j,d_i,e_j\}$.
The line $\ell_4$ separates the subsets $A$ and $B$ from the subsets $D$ and $E$. Thus, $\ell_4$ separates the edge $a_ib_j$ from the edge $d_ie_j$.
The line $\ell_5$ separates the subsets $A$ and $E$ from the subsets $B$ and $D$. Thus, $\ell_5$ intersects the edges $a_ib_j$ and $d_ie_j$, of $P$.
Consider the segment of $\ell_5$ defined by its intersection point with $a_ib_j$, and by its intersection point with $d_ie_j$; call this segment $s$. As $\ell_4$ lies between $a_ib_j$ and $d_ie_j$, the point of intersection of $\ell_5$ with $\ell_4$ (which is the point $p$), lies in the interior of  $s$. Furthermore, as $\ell_5$ intersects $P$ exactly twice, $s$ is in the interior of $P$ and thus, $p$ is inside $P$.

Analogously, $p$ is inside the polygons induced by $\mathsf{Y}'_{i,j}$ and $\mathsf{Z}'_{i,j}$.

\end{proof}
\end{lemma}

\begin{lemma}\label{lem:5-partite-intersect}
  For each pair of graphs from the set $\{\mathsf{X}_{i,j}, \mathsf{Y}_{i,j},
  \mathsf{Z}_{i,j}\}$ there exists a pair of edges, one of each graph, which intersect.
\begin{proof}

We prove by contradiction. Assume that $\mathsf{Q}$ and $\mathsf{R}$ are two different graphs in $\{\mathsf{X}_{i,j},\mathsf{Y}_{i,j},\mathsf{Z}_{i,j}\}$ that do not intersect. By Lemma~\ref{lem:p_inside}, $p$ is inside both polygons, $P_\mathsf{Q}$ and $P_\mathsf{R}$, induced by $\mathsf{Q}$ and $\mathsf{R}$, respectively. Since, by assumption, $\mathsf{Q}$ and $\mathsf{R}$ do not intersect, $P_\mathsf{Q}$ and $P_\mathsf{R}$ do not intersect either, and one has to be contained inside the other (as both contain $p$). Without loss of generality let $P_\mathsf{Q}$ be inside $P_\mathsf{R}$. One edge of $\mathsf{Q}$ is connecting $P_\mathsf{Q}$ (in the interior of $P_\mathsf{R}$) with a vertex in $G$ (in the exterior of $P_\mathsf{R}$) and therefore intersecting an edge of $\mathsf{R}$. This is a contradiction to the assumption, and the theorem follows.

\end{proof}
\end{lemma}

We can now end the proof of our main theorem.

\begin{proof}[\textbf{Proof of Theorem~\ref{thm:big} ii)}]

  For every geometric graph $\mathsf{K}_n$ of $K_n$ with $n > 18$, one can construct the
  configuration $\mathcal{L}$. Using $\mathcal{L}$ we choose the edge disjoint graphs  $\mathsf{G}_{i,j}$
  ($\mathsf{G}_{i,j}\in\{\mathsf{X}_{i,j},\mathsf{Y}_{i,j},\mathsf{Z}_{i,j}\}$). By construction $\mathsf{K}_n$
  contains $\frac{12}{169}n^2-\Theta(n)$ of these graphs, and we assign a different color to each of them. By
  Lemma~\ref{lem:5-partite-intersect} each two of these subgraphs
  intersect, therefore $0.0710n^{2}-\Theta(n)\leq\psi_{g}(K_{n})$.
\end{proof}

\section*{Acknowledgments}

Part of the work was done during the $4^{th}$ Workshop on Discrete Geometry and its Applications, held at Centro de Innovaci\'on Matem\'atica, Juriquilla, Mexico, February 2012. We thank Marcelino Ram\'irez-Iba\~nez and all other participants for useful discussions.

O.A. partially supported by the ESF EUROCORES programme EuroGIGA - ComPoSe, Austrian Science Fund (FWF): I~648-N18. G.A. partially supported by CONACyT of Mexico, grant 166306; and PAPIIT of Mexico, grant IN101912. T.H. supported by the Austrian Science Fund (FWF): P23629-N18 `Combinatorial Problems on Geometric Graphs'. D.L. partially supported by CONACYT of Mexico, grant 153984.

\bibliographystyle{amsplain}
\bibliography{biblio}
\end{document}